\documentclass{amsart}
\usepackage[utf8]{inputenc}
\usepackage{amsmath}
\usepackage{amsthm}
\usepackage{amsfonts, dsfont}
\usepackage{amssymb}
\usepackage[all]{xy}
\usepackage{indentfirst}
\usepackage[pagebackref=true]{hyperref}

\usepackage{cleveref}
\usepackage[alphabetic, initials]{amsrefs}

\usepackage{faktor}
\usepackage{xfrac}

\newtheorem{thm}{Theorem}[section]

\newtheorem{theorem}[thm]{Theorem}
\newtheorem{corollary}[thm]{Corollary}
\newtheorem{proposition}[thm]{Proposition}

\newtheorem{lemma}[thm]{Lemma}
\newtheorem*{theorem*}{Theorem}
\newtheorem*{corollary*}{Corollary}

\theoremstyle{definition}

\newtheorem*{defn*}{Definiton}
\newtheorem{example}[thm]{Example}

\newtheorem*{ack}{Acknowledgements}
\newtheorem{remark}[thm]{Remark}

\newcommand{\N}{\mathbb{N}} %% Naturals
\newcommand{\Z}{\mathbb{Z}} %% Integers
\usepackage{verbatim} %% Enables 'comment'
\usepackage{color}
\newcommand{\G}{\Gamma}

\newcommand{\La}{\Lambda}

\newcommand{\Id}{\mathrm{Id}}

\newcommand{\Sub}{\operatorname{Sub}}
\newcommand{\Stab}{\operatorname{Stab}}

\newcommand{\act}{\!\curvearrowright\!}

\newcommand{\st}{\operatorname{St}}

\newcommand{\sA}{\mathsf{A}}
\newcommand{\sF}{\mathsf{F}}

\newcommand{\cB}{\mathcal{B}}

\newcommand{\cF}{\mathcal{F}}

\newcommand{\cP}{\mathcal{P}}

\newcommand{\cU}{\mathcal{U}}

\newcommand{\one}{\boldsymbol{1}}

\title{A dichotomy for topological full groups}
\author{Eduardo Scarparo}

\address{Eduardo Scarparo\\ University of Glasgow\\ United Kingdom}
\email{eduardo.scarparo@glasgow.ac.uk}

\thanks{This project has received funding from the European Research Council (ERC) under the European Union's Horizon 2020 research and innovation programme (grant agreement No. 817597).}

\begin{document}

%%%%%%%%%%%%%%%%

\begin{abstract}
Given a minimal action $\alpha$ of a countable group on the Cantor set, we show that the alternating full group $\sA(\alpha)$ is non-amenable if and only if the topological full group $\sF(\alpha)$ is $C^*$-simple.  This implies, for instance, that the Elek-Monod example of non-amenable topological full group coming from a Cantor minimal $\Z^2$-system is $C^*$-simple.

\end{abstract}
\maketitle
\section{Introduction}

Given an action $\alpha$ of a group on the Cantor set $X$, the \emph{topological full group} of $\alpha$, denoted by $\mathsf{F}(\alpha)$, is the group of homeomorphisms on $X$ which are locally given by $\alpha$. 

In \cite{JM13}, Juschenko and Monod showed that topological full groups of Cantor minimal $\Z$-systems are amenable.  Together with results of Matui (\cite{Mat06}), this gave rise to the first examples of infinite, simple, finitely generated, amenable groups. On the other hand,  in \cite{EM13}, Elek and Monod constructed an example of a free minimal $\Z^2$-subshift whose topological full group contains a free group. 

A group $\G$ is said to have the \emph{unique trace property} if its reduced $C^*$-algebra $C^*_r(\G)$ has a unique tracial state and to be \emph{$C^*$-simple} if $C^*_r(\G)$ is simple.  In \cite{BKKO},  Breuillard,  Kalantar,  Kennedy and Ozawa showed that $\G$ has the unique trace property if and only if it does not contain any non-trivial amenable normal subgroup, and in \cite{Ken20} Kennedy showed that $\G$ is $C^*$-simple if and only if it does not contain any non-trivial amenable uniformly recurrent subgroup. 

By using this new characterization of $C^*$-simplicity,  Le Boudec and Matte Bon showed in  \cite{LBMB18} that the topological full group of a free minimal action of a countable non-amenable group on the Cantor set is $C^*$-simple, and asked whether the same conclusion holds if one does not assume freeness.  In \cite{BS19}, Brix and the author showed that it suffices to assume that the action is topologically free. In \cite{KTD19}, Kerr and Tucker-Drob obtained examples of $C^*$-simple topological full groups coming from actions of amenable groups.

Given an action $\alpha$ of a group on the Cantor set, Nekrashevych introduced in \cite{Nek19} the \emph{alternating full group} of the action,  which we denote by $\sA(\alpha)$. This is a normal subgroup of $\sF(\alpha)$ generated by certain copies of finite alternating groups. It was shown in \cite{Nek19} that if $\alpha$ is minimal, then $\sA(\alpha)$ is simple and is contained in every non-trivial normal subgroup of $\sF(\alpha)$.

In \cite{MB18}, Matte Bon obtained a classification of uniformly recurrent subgroups of topological full groups. By using this result, we show the following:
\begin{theorem*}[Theorem \ref{thm:main}]
Let $\alpha$ be a minimal action of a countable group on the Cantor set.  The following conditions are equivalent:
\begin{enumerate}
\item[(i)] $\sA(\alpha)$ is non-amenable;
\item[(ii)] Any group $H$ such that $\sA(\alpha)\leq H\leq\sF(\alpha)$ is $C^*$-simple;
\item[(iii)] There exists a $C^*$-simple group $H$ such that $\sA(\alpha)\leq H\leq\sF(\alpha)$.
\end{enumerate}
\end{theorem*}
As a consequence, we obtain the following:
\begin{corollary*}[Corollary \ref{cor:der}]
Let $\alpha$ be a minimal action of a countable group on the Cantor set.  Then $\sF(\alpha)$ has the unique trace property if and only if it is $C^*$-simple. If $\sA(\alpha)=\sF(\alpha)'$, then $\sF(\alpha)$ is non-amenable if and only if it is $C^*$-simple. 
\end{corollary*}

It is still an open problem whether $\sA(\alpha)$ always coincides with $\sF(\alpha)'$, but in many cases this is known to be true.  For example, it follows from results of Matui (\cite{Mat15}) that this is the case for free Cantor minimal $\Z^n$-systems.  This implies that the example of non-amenable topological full group coming from an action of $\Z^2$ in \cite{EM13} is $C^*$-simple.
\begin{ack}
I am grateful to Eduard Ortega and Nicolás Matte Bon for helpful comments.  I also thank the anonymous referee for suggestions which helped to improve the presentation of this work and for pointing out an incorrection in the original formulation of Proposition \ref{prop:mea}.
\end{ack}
\section{Preliminaries}

\subsection*{Topological dynamics}Given a locally compact Hausdorff space $X$, we denote by $\cB(X)$ the Borel $\sigma$-algebra of $X$, and by $\cP(X)$ the space of regular probability measures on $X$.  

If $\G$ is a group acting by homeomorphisms on $X$, we say that $X$ is a \emph{locally compact $\G$-space}.  If $X$ admits no non-trivial $\G$-invariant closed subspaces, then we say that $X$ (or the action) is \emph{minimal}. Given $U\subset X$, let $\st_\G(U)$ consist of the elements of $\G$ which fix pointwise $U$, and $\st_\G(U)^0$ consist of the elements of $\G$ which fix pointwise a neighborhood of $U$.  To ease the notation, given $x\in X$, we let $\G_x:=\st_\G(\{x\})$ and $\G_x^0:=\st_\G(\{x\})^0$.

Denote by $\Sub(\G)$ the set of subgroups of $\G$, endowed with the \emph{Chabauty topology}; this is the restriction to $\Sub(\G)$ of the product topology on $\{0, 1\}^\G$, where 
every subgroup $\La\in \Sub(G)$ is identified with its characteristic function $\one_\La \in \{0, 1\}^\G$.  Notice that the space of amenable subgroups $\Sub_{am}(\G)$ is closed in $\Sub(\G)$. We consider $\Sub(\G)$ as a compact $\G$-space under the action by conjugation.  A subgroup $\La\leq\G$ is said to be \emph{confined} if  $\{e\}$ is not in the closure of the $\G$-orbit of $\La$.

An \emph{invariant random subgroup} (IRS) is a $\G$-invariant regular probability measure on $\Sub(\G)$. We say an IRS is \emph{amenable} if its support is contained in $\Sub_{am}(\G)$. By \cite[Corollary 4.3]{BKKO} and \cite[Corollary 1.5]{BDL16}, $\G$ has the unique trace property if and only if its unique amenable normal subgroup is $\{e\}$, if and only if its unique amenable IRS is $\delta_{\{e\}}$. 

A \emph{uniformly recurrent subgroup} (URS) is a $\G$-invariant closed minimal subspace $\cU\subset\Sub(\G)$. We say $\cU$ is \emph{amenable} if every element of $\cU$ is amenable. By \cite[Theorem 4.1]{Ken20}, $\G$ is $C^*$-simple if and only if its only amenable URS is $\{\{e\}\}$. Alternatively, $\G$ is $C^*$-simple if and only if it does not contain any confined amenable subgroup.

Suppose $\G$ is countable and $X$ is a minimal compact $\G$-space.  Let $\Stab_\G\colon X\to\Sub(\G)$ be the map given by $\Stab_\G(x):=\G_x$ and $\Stab_\G^0\colon X\to \Sub(\G)$ be the map given by $\Stab_\G^0(x):=\G_x^0$, for $x\in X$. Notice that $\Stab_\G$ and $\Stab_\G^0$ are Borel measurable and $\G$-equivariant. Moreover, the set $Y:=\{x\in X:\G_x=\G_x^0\}$ is dense in $X$ and $\Stab_\G(Y)$ is a URS, the so called \emph{stabilizer URS} of the action $\G\act X$ (for a proof of these last claims, see \cite[Section 2]{LBMB18}).

\subsection*{Topological full groups}

Fix an action $\alpha$ of a group $\G$ on the Cantor set $X$.  We say that a homeomorphism $h\colon U\to V$ between clopen subsets $U,V\subset X$ is \emph{locally given by $\alpha$} if there exist $g_1,\dots,g_n\in\G$ and clopen sets $A_1,\dots,A_n\subset U$ such that $U=\bigsqcup_{i=1}^n A_i$ and $h|_{A_i}=g_i|_{A_i}$ for $1\leq i \leq n$. The \emph{topological full group} of $\alpha$, denoted by $\mathsf{F}(\alpha)$, is the group of homeomorphisms $h\colon X\to X$ which are locally given by $\alpha$.

Given $d\in\N$, a \emph{$d$-multisection} is a collection of $d$ disjoint clopen sets $(A_i)_{i=1}^d\subset X$ and $d^2$ homeomorphisms $(h_{i,j}\colon A_i\to A_j)_{i,j=1}^d$ which are locally given by $\alpha$ and such that, for $1\leq i,j,k\leq d$,  it holds that $h_{j,k}h_{i,j}=h_{i,k}$ and $h_{i,i}=\Id_{A_i}$.

Given $d\in \N$, let $S_d$ and $A_d$ be the symmetric and alternating groups, respectively. Given a $d$-multisection $\cF=((A_i)_{i=1}^d,(h_{i,j})_{i,j=1}^d)$ and $\sigma\in S_d$, let $\cF(\sigma)\in\sF(\alpha)$ be given by $\cF(\sigma)|_{A_i}:=h_{\sigma(i),i}$, for $1\leq i \leq n$ and $\cF(\sigma)(x)=x$ for $x\notin\bigsqcup_{i=1}^n A_i$. The \emph{alternating full group} $\sA(\alpha)$ is the subgroup of $\sF(\alpha)$ generated by 
$$\{\cF(\sigma):\text{$d\in\N$, $\cF$ is a $d$-multisection, $\sigma\in A_d$}\}.$$ 

Notice that $\sA(\alpha)$ is normal in $\sF(\alpha)$ and that $\sA(\alpha)$ is contained in the derived subgroup $\sF(\alpha)'$. If $\alpha$ is a minimal action of a countable group on the Cantor set,  then $\sA(\alpha)$ is simple (\cite[Theorem 4.1]{Nek19}).

\begin{remark}
Alternatively, $\sF(\alpha)$ and $\sA(\alpha)$ can be described as groups of bisections of the groupoid of germs of $\alpha$. This is the point of view adopted in \cite{Nek19} and \cite{MB18}. Conversely, given a effective groupoid $G$ with unit space $G^{(0)}$ homeomorphic to the Cantor set, denote by $\alpha$ the natural action of the topological full group of $G$ on $G^{(0)}$.  Then the topological and alternating full groups of $G$ coincide with $\sF(\alpha)$ and $\sA(\alpha)$, respectively (\cite[Corollary 4.7]{NO19}).
\end{remark}

\section{$C^*$-simplicity of full groups}

Given a locally compact $\G$-space $X$ and $U\subset X$ open not necessarily invariant, let 
$$
\cP_\G(U):=\{\mu\in\cP(U):\forall g\in\G\  \forall A\in\cB(U),\ \ \mu(A\cap g^{-1}U)=\mu(gA\cap U)\}.
$$

Alternatively, one can characterize $\cP_\G(U)$ as the measures $\mu\in\cP(U)$ such that $\mu(gA)=\mu(A)$ for every $g\in\G$ and $A\in\cB(U)$ such that $gA\subset U$.

\begin{proposition}\label{prop:mea}
Let $X$ be a compact $\G$-space and $U\subset X$ open such that $X=\G. U$. Then the map $j\colon\cP_\G(X)\to\cP_\G(U)$ given by $j(\nu):=\frac{\nu|_{\cB(U)} }{\nu(U)}$ is a well-defined bijection.
\end{proposition}
\begin{proof}
Take $g_1,\dots,g_n\in\G$ such that $X=\bigcup_{i=1}^n g_iU$.  For $1\leq i \leq n$, let 
$$A_i:=U\setminus\bigcup_{j=1}^{i-1}g_i^{-1}g_j U.$$ 
Then $X=\bigsqcup_{i=1}^n g_iA_i$. Given $\nu\in\cP_\G(X)$, obviously $\nu(U)\geq1/n$, so that $j$ is a well-defined map. Moreover, given $A\in\cB(X)$, we have 
\begin{equation}\label{eq:dec}
\nu(A)=\sum_{i=1}^n\nu(g_i A_i\cap A)=\sum_{i=1}^n\nu(A_i\cap g_i^{-1} A).
\end{equation}
Since each $A_i$ is contained in $U$, this implies that $\nu$ is determined by its restriction to $\cB(U)$. 

If $\nu_1,\nu_2\in\cP_\G(X)$ are such that $j(\nu_1)=j(\nu_2)$, then $\nu_1|_U=\frac{\nu_1(U)}{\nu_2(U)}\nu_2|_U$.  Furthermore,  by \eqref{eq:dec},  we have
$$1=\nu_1(X)=\sum_{i=1}^n\nu_1(A_i)=\sum_{i=1}^n\frac{\nu_1(U)}{\nu_2(U)}\nu_2(A_i)=\frac{\nu_1(U)}{\nu_2(U)}\nu_2(X)=\frac{\nu_1(U)}{\nu_2(U)},$$
hence $\nu_1(U)=\nu_2(U)$. Consequently,  $\nu_1=\nu_2$ and $j$ is injective.

Let us now show that $j$ is surjective. Given $\mu\in\cP_\G(U)$ and $A\in\cB(X)$, let 
$$\nu(A):=\sum_{i=1}^n\mu(A_i\cap g_i^{-1}A).$$

Given $B\in \cB(U)$, we have

$$\nu(B)=\sum_{i=1}^n\mu(A_i\cap g_i^{-1}B)=\sum_{i=1}^n\mu(g_iA_i\cap B)=\mu(B),$$
so that $\nu|_{\cB(U)}=\mu$. 

We claim that $\nu$ is $\G$-invariant. Fix $A\in \cB(X)$ and $g\in\G$, and we will show that $\nu(A)=\nu(gA)$.  

For $1\leq i\leq n$, let $h_i:=g^{-1}g_i$,  $B_i:=A_i\cap g_i^{-1}A$ and $C_i:=A_i\cap g_i^{-1}gA=A_i\cap h_i^{-1} A$. By definition of $\nu$ we have $\nu(A)=\sum_{i=1}^n\mu(B_i)$ and $\nu(gA)=\sum_{i=1}^n\mu(C_i)$. 

Moreover, one can readily check that $A=\bigsqcup_{i=1}^n g_i B_i=\bigsqcup _{i=1}^n h_iC_i$.

For $1\leq i,j\leq n$, let $B_{i,j}:=B_i\cap g_i^{-1}h_jC_j$ and $C_{i,j}:=h_j^{-1}g_iB_i\cap C_j$. 

Notice that, for $1\leq i\leq n$, 
$$\bigsqcup_{j=1}^n B_{i,j}=B_i\cap g_i^{-1}A=B_i$$
and, for $1\leq j \leq n$,
$$\bigsqcup_{i=1}^n C_{i,j}=h_j^{-1}A\cap C_j=C_j.$$

Furthermore, $g_i B_{i,j}=h_jC_{i,j}$,  hence $\mu(B_{i,j})=\mu(C_{i,j})$, since $B_{i,j}$ and $C_{i,j}$ are contained in $U$ for every $i,j$. Therefore,  
\begin{align*}
\nu(A)=\sum_{i=1}^n\mu(B_i)=\sum_{i,j=1}^n \mu(B_{i,j})=\sum_{i,j=1}^n\mu(C_{i,j})=\sum_{j=1}^n\mu(C_j)=\nu(gA).
\end{align*}

Finally, we have that $j(\frac{\nu}{\nu(X)})=\frac{\nu|_{\cB(U)}/\nu(X)}{\nu(U)/\nu(X)}=\nu|_{\cB(U)}=\mu.$
\end{proof}
\begin{remark}
Let $\G\act X$ and $\La\act Y$ be actions on compact spaces. The actions are said to be \emph{Kakutani equivalent} \cite[Definition 2.14]{Li18} if there exist clopen sets $A\subset X$ and $B\subset Y$ such that $X=\G .A$, $Y=\La . B$ and the partial transformation groupoids obtained by restriction to $A$ and $B$ are isomorphic. Proposition \ref{prop:mea} implies that Kakutani equivalence induces a bijection between $\cP_\G(X)$ and $\cP_\La(X)$.
\end{remark}

The proof of the following result is analogous to \cite[Lemma 4.9.(2)]{NO19}.
\begin{lemma}\label{lem:cover}

Let $\alpha$ be a minimal action of a group $\G$ on the Cantor set $X$.  Given $U\subset X$ clopen,  $x\in U$ and $g\in\G$ such that $g(x)\in U$,  there exists a neighborhood $V$ of $x$ and $h\in\st_{\sA(\alpha)}(U^\mathsf{c})$ such that $g|_V=h|_V$.
\end{lemma}
\begin{proof}

Case 1: $g(x)\neq x$. Take $k\in \G$ such that $k(g(x))\in U\setminus\{x,g(x)\}$. Let $V$ be a clopen neighborhood of $x$ such that $V$, $g(V)$ and $kg(V)$ are disjoint subsets of $U$. Then the homeomorphisms  $h_{2,1}:=g|_V$ and $h_{3,1}:=kg|_V$ give rise to a $3$-multisection $\cF$ such that $\cF((123))|_V=g|_V$ and $\cF((123))\in\st_{\sA(\alpha)}(U^\mathsf{c})$.

Case 2: $g(x)=x$. Take $k\in \G$ such that $k(x)\in U\setminus\{ x\}$. By case 1, there are $h_1,h_2\in\st_{\sA(\alpha)}(U^\mathsf{c})$ and $V_1,V_2$ neighborhoods of $x$ and $k(x)$, respectively,  such that $k|_{V_1}=h_1|_{V_1}$ and $gk^{-1}|_{V_2}=h_2|_{V_2}$. Then $V:=V_1\cap k^{-1}(V_2)$ is a neighborhood of $x$ such that $h_2h_1|_V=g|_V$.
\end{proof}

The next lemma uses the same idea of \cite[Corollary 6.5]{MB18}. 

\begin{lemma}\label{lem:pre}
Let $\alpha$ be a minimal action of a countable group $\G$  on the Cantor set $X$ and $H$ a group such that $\sA(\alpha)\leq H\leq \sF(\alpha)$. Then $H$ is not $C^*$-simple if and only if $\sA(\alpha)_x^0$ is amenable for all $x\in X$.
\end{lemma}
\begin{proof}
Suppose $H$ is not $C^*$-simple.  Then $H$ contains a confined amenable subgroup. By \cite[Theorem 6.1]{MB18}, there exists $Q\subset X$ finite such that $\st_{\sA(\alpha)}^0(Q)$ is amenable. Given $x\in X$, take a net $(g_i)\subset\sF(\alpha)$ such that $g_iq\to x$ for any $q\in Q$ (existence of such a net $(g_i)$ follows from minimality and proximality of $\sF(\alpha)\act X$, see \cite[Lemma 5.12]{MB18}).  Take $K$ a limit point of $g_i\st_{\sA(\alpha)}^0(Q)g_i^{-1}$. One can readily check that $\sA(\alpha)_x^0\leq K$, hence $\sA(\alpha)_x^0$ is amenable.

Conversely, if $\sA(\alpha)_x^0$ is amenable for all $x\in X$, then, since $\sA(\alpha)_x^0$ is non-trivial for every $x$, it follows that the stabilizer URS $\cU$ of $\sA(\alpha)\act X$ is a non-trivial amenable URS of $\sA(\alpha)$.  By \cite[Theorem 6.1]{MB18},  any element of $\cU$ is a confined subgroup of $\mathsf{F}(\alpha)$ (hence of $H$ as well).  Therefore,  $H$ is not $C^*$-simple.
\end{proof}

\begin{theorem}\label{thm:main}
Let $\alpha$ be a minimal action of a countable group $\G$ on the Cantor set $X$.  The following conditions are equivalent:
\begin{enumerate}
\item[(i)] $\sA(\alpha)$ is non-amenable;
\item[(ii)] Any group $H$ such that $\sA(\alpha)\leq H\leq\sF(\alpha)$ is $C^*$-simple;
\item[(iii)] There exists a $C^*$-simple group $H$ such that $\sA(\alpha)\leq H\leq\sF(\alpha)$.
\end{enumerate}
\end{theorem}
\begin{proof}
The implications (ii)$\implies$(iii)$\implies$(i) are immediate.

(i)$\implies$(ii): Suppose that there exists $H$ non-$C^*$-simple such that $\sA(\alpha)\leq H\leq\sF(\alpha)$.  By Lemma \ref{lem:pre}, $\sA(\alpha)_x^0$ is amenable for every $x\in X$. 

Fix a clopen non-empty set $U$ properly contained in $X$.  Since, for any $x\in U^\mathrm{c}$, we have $\La:=\st_{\sA(\alpha)}(U^\mathsf{c})\leq\sA(\alpha)_x^0$, it follows that $\La$ is amenable. 

Let $\mu\in\cP_\La(U)$ and we claim that $\mu\in\cP_\G(U)$.  By regularity, it suffices to show that, for any $K\subset U$ compact and $g\in \G$ such that $g(K)\subset U$, it holds that $\mu(gK)=gK$. By Lemma \ref{lem:cover}, there are $h_1,\dots,h_n\in\La$ and a partition $K=\bigsqcup_{i=1}^n K_i$ into compact sets such that $g|_{K_i}=h_i|_{K_i}$ for $1\leq i\leq n$.  Therefore, 
$$\mu(gK)=\sum_{i=1}^n\mu(gK_i)=\sum_{i=1}^n\mu(h_iK_i)=\sum_{i=1}^n\mu(K_i)=\mu(K)$$
 and $\mu\in\cP_\G(U)$. 

We conclude from Proposition \ref{prop:mea} that there is $\nu\in\cP_\G(X)=\cP_{\sF(\alpha)}(X)$.  Furthermore, by minimality of the action, $\nu$ has full support.  Let $\rho:=(\Stab_{\sA(\alpha)}^0)_*\nu$.

Given $g\in\La\setminus\{e\}$, we have $\rho(\{K\in\Sub(\sA(\alpha)):g\in K\})\geq\nu(U^\mathsf{c})>0$. Hence, $\rho$ is a non-trivial amenable IRS on $\sA(\alpha)$. Since $\sA(\alpha)$ is simple, this implies that $\sA(\alpha)$ is amenable.

\end{proof}
The following is an immediate consequence of Theorem \ref{thm:main}.

\begin{corollary}\label{cor:der} 
Let $\alpha$ be a minimal action of a countable group on the Cantor set.  Then $\sF(\alpha)$ has the unique trace property if and only if it is $C^*$-simple. If $\sA(\alpha)=\sF(\alpha)'$, then $\sF(\alpha)$ is non-amenable if and only if it is $C^*$-simple.
\end{corollary}

\begin{remark}
If $\alpha$ is an action of a group on the non-compact Cantor set $X$, then the topological full group $\sF(\alpha)$ is the group of homeomorphisms on $X$ which are locally given by $\alpha$ and have compact support. Moreover, $\sA(\alpha)$ is defined by requiring that the domains of the partial homeomorphisms of the multisections to be compact-open (as in \cite[Definition 5.1]{MB18}). By arguing as in \cite[Corollary 6.5]{MB18}, the same conclusion of Theorem \ref{thm:main} and Corollary \ref{cor:der} holds in the non-compact case.
\end{remark}
\begin{example}
It follows from \cite[Lemma 6.3]{Mat12} and \cite[Theorem 4.7]{Mat15} that, given a free minimal action $\alpha$ of $\Z^n$ on the Cantor set, it holds that $\sF(\alpha)'=\sA(\alpha)$. Hence, the example of non-amenable topological full group coming from a Cantor minimal $\Z^2$-system in \cite{EM13} is $C^*$-simple.
\end{example}

\bibliographystyle{acm}
\bibliography{bibliografia}
\end{document}